\documentclass[a4paper,11pt]{amsart}

\usepackage{hyperref}

\usepackage{graphicx}
\usepackage{amssymb}
\usepackage{latexsym}
\usepackage{amsfonts}
\usepackage{amsmath}
\usepackage{amsthm}
\usepackage{amssymb}
\usepackage{mathrsfs}

\newtheorem{theorem}{Theorem}[section]
\newtheorem{lemma}[theorem]{Lemma}
\newtheorem{proposition}[theorem]{Proposition}
\newtheorem{corollary}[theorem]{Corollary}

\newtheorem{example}[theorem]{Example}

\newtheorem{remark}[theorem]{Remark}

\newcommand{\C}{\mbb{C}}

\newcommand{\QQ}{\mc{Q}}

\newcommand{\R}{\mbb{R}}

\newcommand{\mbb}{\mathbb}
\newcommand{\mc}{\mathcal}

\newcommand{\cS}{\mbb{S}}

\newcommand{\q}{\mbb{H}}

\newcommand{\punto}{\boldsymbol{\cdot}}

\begin{document}

\markboth{Riccardo Ghiloni - Alessandro Perotti}
{Lagrange polynomials over Clifford numbers}

%%%%%%%%%%%%%%%%%%%%% Publisher's Area please ignore %%%%%%%%%%%%%%%
%
%\catchline{}{}{}{}{}
%
%%%%%%%%%%%%%%%%%%%%%%%%%%%%%%%%%%%%%%%%%%%%%%%%%%%%%%%%%%%%%%%%%%%%

\title{Lagrange polynomials over Clifford numbers}

\author{RICCARDO GHILONI}

%\address{Department of Mathematics, University of Trento\\ Povo-Trento,  I--38123, Italy}
\email{ghiloni@science.unitn.it}

\author{ALESSANDRO PEROTTI}

%\address{Department of Mathematics, University of Trento\\ Povo-Trento,  I--38123, Italy}
\email{perotti@science.unitn.it}
\thanks{Work partially supported by GNSAGA of INdAM, MIUR-PRIN project ``Variet\`a reali e complesse: geometria, topologia e analisi armonica" and MIUR-FIRB project ``Geometria Differenziale e Teoria Geometrica delle Funzioni"}
\address{Department of Mathematics, University of Trento, I--38123, Povo-Trento, Italy}



\maketitle

%\begin{history}
%\received{(Day Month Year)}
%\revised{(Day Month Year)}
%\accepted{(Day Month Year)}
%\comby{(xxxxxxxxx)}
%\end{history}

\begin{abstract}
We construct Lagrange interpolating polynomials for a set of points and values belonging to the algebra of real quaternions $\q \simeq \R_{0,2}$, or to the real Clifford algebra $\R_{0,3}$. In the quaternionic case, the approach by means of Lagrange polynomials is new, and gives a complete solution of the interpolation problem. In the case of $\R_{0,3}$, such a problem is dealt with here for the first time. Elements of the recent theory of slice regular functions are used. Leaving apart the classical cases $\R_{0,0} \simeq\R$, $\R_{0,1}\simeq\C$ and the trivial case $\R_{1,0} \simeq \R \oplus \R$, the interpolation problem on Clifford algebras $\R_{p,q}$ with $(p,q) \neq (0,2),(0,3)$ seems to have some intrinsic difficulties. 
\end{abstract}

\section{Introduction}
The aim of this work is to define Lagrange interpolating polynomials for a set of points and values belonging to a real Clifford algebra. We make some preliminary considerations to select the Clifford algebras on which the construction can be performed. We then restrict to two cases, the Clifford algebra of signature $(0,2)$ (isomorphic to the algebra of real quaternions) and the one of signature $(0,3)$. 

Let $\R_{p,q}$ denote the real Clifford algebra with signature $(p,q)$, equipped with the usual Clifford anti-involution $x\mapsto x^c$ defined by
\[
x^c=([x]_0+[x]_1+[x]_2+[x]_3+[x]_4+\cdots)^c=[x]_0-[x]_1-[x]_2+[x]_3+[x]_4-\cdots,
\]
where $[x]_k$ denotes the $k$--vector component of $x\in \R_{p,q}$ (cf.\ e.g.\ \cite[\S4.1]{CSSSbook} or \cite[\S 3.2]{GHS}). 
For every element $x$ of $\R_{p,q}$, the \emph{trace} of $x$ is $t(x):=x+x^c$ and the (squared) \emph{norm} of $x$ is
$n(x):=xx^c$.  
Let $m:=p+q$. An element $x$ of $\R_{p,q}$ can be represented in the form $x=\sum_K x_Ke_K$, with $K=(i_1,\ldots,i_k)$ an increasing multiindex of length $k$, $0\le k \le m$, $e_K=e_{i_1}\cdots e_{i_k}$, $e_\emptyset=1$,  $x_K\in\R$, $x_\emptyset=x_0$, $e_1,\ldots,e_m$ basis elements (with $e_i^2=1$ for $i\le p$,  $e_i^2=-1$ for $i>p$). The (real vector) subspace  generated by $1,e_1,e_2,\ldots,e_m$ is called the set of \emph{para\-vectors} in $\R_{p,q}$ and denoted by $\R^{(m+1)}$. We identify the field of real numbers with the subspace of $\R_{p,q}$ generated by the unit of the algebra.

In a non--commutative setting, the ring of polynomials is usually defined by fixing the position of the coefficients w.r.t.\  the indeterminate $X$ (e.g.\ on the right) and by imposing commutativity of $X$ with the coefficients when two polynomials are multiplied together (cf.\ e.g.\ \cite[\S 16]{Lam}). Given two polynomials $P(X)$ and $Q(X)$, let $P\punto Q$ denote the product obtained in the way we just described. %This product was extended to every pair of slice functions in \cite{,}, and called  \emph{slice product}.
If $P$ has \emph{real} coefficients, then $(P\punto Q)(x)=P(x)Q(x)$.
In general, a direct computation (cf.~\cite[\S 16.3]{Lam}) shows that if $P(x)$ is invertible, then
\begin{equation}\label{product}
(P\punto Q)(x)=P(x)Q(P(x)^{-1}xP(x)).
\end{equation}

In this setting, a {(left) root} of a polynomial $P(X)=\sum_{h=0}^dX^h a_h$ is an element $x\in\R_{p,q}$ such that $P(x)=\textstyle\sum_{h=0}^dx^h a_h=0$. As shown in \cite{GhPe_AIM} and \cite{GhPe_Trends}, in order to obtain a good structure for the zero locus of a polynomial,  it is necessary to restrict the domain where roots are looked for, and to impose some conditions on the polynomial. We recall from \cite{GhPe_AIM} the definition %of the \emph{normal cone} $\NN_{p,q}$ and 
of the \emph{quadratic cone} $\QQ_{p,q}$ of $\R_{p,q}$:
%\begin{align*}
%\NN_{p,q}&=\{0\}\cup\{x\in\R_{p,q}\;|\; n(x)=n(x^c)\in\R\setminus\{0\}\},\\ 
\[\QQ_{p,q}:=\R\cup\{x\in\R_{p,q}\;|\; t(x)\in\R,n(x)\in\R,4n(x)>t(x)^2\}.\]
%\end{align*}
The quadratic cone coincides with the whole Clifford algebra  only when $\R_{p,q}$ is a division algebra, i.e.\ for $\R_{0,0}\simeq\R$, $\R_{0,1}\simeq\C$ and $\R_{0,2}\simeq\q$.
If we restrict roots to the quadratic cone, then an \emph{admissible} polynomial $P(X)$ with Clifford coefficients (cf.~\cite{GhPe_AIM} for this notion) satisfies a version of the Fundamental Theorem of Algebra.
%In particular, the result holds for polynomials  with paravector coefficients, a case considered in \cite{YangQianActa}.

In the construction of Lagrange polynomials for points $x_1,\ldots, x_n$, we are lead to consider only elements $x_i$ in the quadratic cone.
%\[V(P)=\{x\in\QQ_{p,q}\;|\; P(x)=\textstyle\sum_{h=0}^dx^h a_h=0\},\] 
Moreover, the procedure requires the invertibility of differences of the form $x_i-x_j$ ($i\ne j$), and also of the form $x_i'-x_j'$, with $x_i'$ and $x_j'$ in the same conjugacy classes of $x_i$ and $x_j$, respectively.

These conditions impose severe restrictions on the Clifford algebras in which the procedure can be done. In every Clifford algebra $\R_{p,q}$ with $p\ge2$, one can find elements $x,y\in \QQ_{p,q}$, not belonging to the same conjugacy class, such that the difference $x-y$ is non--invertible (e.g.\ $x=e_{12}$ and $y=1/3e_1+2/3e_{12}$).
Due to the isomorphism between $\R_{2,0}$ and $\R_{1,1}$, also in the algebras $\R_{1,q}$, with $q\ge1$, one can find pairs of elements with the same properties. As we will see below, this fact has consequences also on the number of roots of polynomials, and therefore on the uniqueness of interpolating polynomials.

In $\R_{1,0}\simeq\R\oplus\R$,  the quadratic cone reduces to the real line $\R$, where the construction of the Lagrange polynomials is well--known. Therefore, leaving apart the classical cases $\R_{0,0}\simeq\R$, $\R_{0,1}\simeq\C$, we are left with the algebras $\R_{0,q}$, with $q\ge2$. In this case, the quadratic cone is simply
\[\QQ_{0,q}=\R\cup\{x\in\R_{0,q}\;|\; t(x)\in\R,n(x)\in\R\}.\]
Let $\cS$ denote the set of square roots of $-1$ in $\QQ_{0,q}$. $\cS$ is the set of elements $J$ such that $t(J)=0$, $n(J)=1$.  Every $x\in\QQ_{0,q}$ has a decomposition $x=\alpha+\beta J$ with $\alpha,\beta\in\R$, $\beta\ge0$ and $J\in\cS$. 

For $q=2,3$, it can be shown that the conjugacy class of $y=\alpha+\beta J\in\QQ_{0,q}$ is the set $\cS_y=\{\alpha+\beta K\;|\;K\in\cS\}\subset\QQ_{0,q}$. This comes from the fact that $\cS$ forms a unique conjugacy class (cf.\ \cite{HHA}).  Since $t(y)=2\alpha$, $n(y)=\alpha^2+\beta^2$, $\cS_x=\cS_y$ if and only if $t(x)=t(y)$ and $n(x)=n(y)$.  The set $\cS_y$  coincides with the zero locus in $\QQ_{0,q}$  of the \emph{characteristic polynomial} of $y$, i.e.\ the polynomial with real coefficients
\[\Delta_y(X):=(X-y)\cdot (X-y^c)=X^2-X t(y)+n(y).\]
If $q\ge4$, the conjugacy class of an element $x\in\QQ_{0,q}$ is not necessarily contained in the quadratic cone. For example, the class of $e_4$ contains also $(2+e_{123})^{-1}e_4(2+e_{123})=5/3e_4-4/3 e_{1234}\notin\QQ_{0,q}$.
For this reason, we will restrict to the cases $q=2,3$.

\section{Preliminary results}

For $q=2$ (the quaternionic case), the trace and the norm of an element are always real. For $q=3$, $t$ and $n$ take values in the center of the algebra, i.e.\ in the subspace generated by $1$ and $e_{123}$: $t(x)=2(x_0+x_{123})$, $n(x)=|x|^2 +\phi(x)e_{123}$, where $|x|=(x\cdot x)^{1/2}$ is the euclidean norm of $x\in\R_{0,3}\simeq\R^8$ and $\phi(x)=x\cdot (xe_{123})=2(x_0x_{123}-x_1x_{23}+x_2x_{13}-x_3x_{12})$. Since $\phi(x)=\phi(x^c)$, we have that $n(x)=n(x^c)$ for each $x\in\R_{0,3}$.
We have $\QQ_{0,2}=\R_{0,2}\simeq\q$,  $\QQ_{0,3}=\{x\in\R_{0,3}\;|\; x_{123}=0,\phi(x)=0\}$. 

%This property is no longer true in higher dimensional Clifford algebras $\R_n$: when $n>3$, the conjugacy class of $x\in\QQ_n$ can be larger than $\cS_x$.

In $\q$, every non--zero element is invertible. In $\R_{0,3}$, there are non--invertible non--zero elements.

\begin{proposition}\label{pro1}
Let $x\in\R_{0,3}$. Then the following facts hold:
\begin{enumerate}
\item[$1.$]
$x$ is invertible if and only if its norm $n(x)$ is invertible.
\item[$2.$]
$x$ is invertible if and only if $\psi_+(x)\psi_-(x)\ne0$, where
\[\psi_\pm(x)=(x_0\pm x_{123})^2+(x_1\mp x_{23})^2+(x_2\pm x_{13})^2+(x_3\mp x_{12})^2.
\]
\end{enumerate}
\end{proposition}

\begin{proof}
1.\ If $x$ is invertible, then $(x^{-1})^c=(x^c)^{-1}$ and $(x^{-1})^cx^{-1}$ is the inverse of $n(x)=xx^c$. Conversely, if $n(x)$ is invertible, then $n(x)^{-1}x^c=x^c n(x)^{-1}$ is the inverse of $x$.

2.\ The center of $\R_{0,3}$ is isomorphic to the algebra $\R\oplus\R$. Therefore, an element $a+be_{123}$ is invertible if and only if $a^2-b^2\ne0$. From the first part it follows that $x\in\R_{0,3}$ is invertible if and only if $|x|^4-\phi(x)^2\ne0$. The thesis follows from the fact that  
$\psi_+(x)=\frac12|x+xe_{123}|^2=|x|^2+\,x\cdot (xe_{123})=|x|^2+\phi(x)$ and $\psi_-(x)=|x|^2-\phi(x)$.
\end{proof}

 In $\QQ_{0,3}$, there exist distinct elements whose difference is not invertible. For example, $e_1,e_{23}$ belong to $\cS$ but $e_1-e_{23}$ is not invertible.

\begin{proposition}\label{invertible}
Let $x,y\in\QQ_{0,3}$, with $\cS_x\ne\cS_y$. Then $x-y$ is invertible.
\end{proposition}
\begin{proof}
If $x-y$ is not invertible, from Proposition~\ref{pro1} we get $\psi_+(x-y)=0$ or $\psi_-(x-y)=0$. Assume $\psi_+(x-y)=0$. Then $x_0+x_{123}=y_0+y_{123}$ and $\psi_+(x)=\psi_+(y)$. Since $x,y\in\QQ_{0,3}$, $\phi(x)=\phi(y)=0$. It follows that $n(x)=|x|^2=\psi_+(x)=n(y)$. 
Moreover,  $x_{123}=y_{123}=0$. Therefore $t(x)=2x_0=2y_0=t(y)$. But then $x$ and $y$ are in the same conjugacy class. The same conclusion is obtained if $\psi_-(x-y)=0$.
\end{proof}

\begin{remark}
The previous results can be obtained also by using an explicit form of the isomorphism $\R_{0,3}\simeq\q\oplus\q$ (cf.\ for example \cite{GHS} for such an isomorphism).
\end{remark}

Let $V(P)$ denote the set of roots of a polynomial $P(X)=\sum_{h=0}^dX^h a_h$ belonging to the quadratic cone $\QQ_{0,q}$:
\[V(P)=\{x\in\QQ_{0,q}\;|\; P(x)=\textstyle\sum_{h=0}^dx^h a_h=0\}.\]

We now prove an analogue of the Gordon--Motzkin Theorem (\cite{GordonMotzkin}, see also \cite[\S 16.4]{Lam}), which in his original form is valid for polynomials over division rings (e.g.\ over the quaternions).

\begin{theorem}\label{GordonMotzkin}
Let $P(X)=\sum_{h=0}^dX^h a_h$ be a polynomial of positive degree $d$, with coefficients $a_h\in\R_{0,3}$. Then its roots  belong to at most $d$ distinct conjugacy classes in $\QQ_{0,3}$.
\end{theorem}

\begin{proof}
We proceed by induction on $d$. If $d=1$, and $xa_1+a_0=ya_1+a_0=0$, with $x,y\in\QQ_{0,3}$, then  $(x-y)a_1=0$. From Proposition~\ref{invertible}, it follows that $x$ and $y$ must belong to the same conjugacy class. For $d\ge2$, let $y\in V(P)$.
Applying the non--commutative version of the Remainder Theorem (cf.~\cite[\S 16.2]{Lam}), we can find a polynomial $Q(X)=\sum_{h=0}^{d-1}X^hb_h$, of degree $d-1$, such that
\[P(X)=(X-y)\punto Q(X).\]
If $x\in V(P)$, with $\cS_x\ne\cS_y$, then $a:=x-y$ is invertible from Proposition~\ref{invertible}. Since $0=P(x)=aQ(a^{-1}xa)$, we get $Q(a^{-1}xa)=0$. 
Therefore $x':=a^{-1}xa\in\cS_x\cap V(Q)$.
From the inductive hypothesis, $x$ belongs to the union of at most $d-1$ conjugacy classes. Therefore the roots of $P$ belong to at most $d$ conjugacy classes.
\end{proof}

It is not clear if the preceding theorem holds on every Clifford algebra $\R_{0,q}$.  The same proof can not be repeated, since for $q>3$ the element $x'=a^{-1}xa$ does not necessarily belongs to the quadratic cone, and then Proposition~\ref{invertible} can not be applied. Surely the result is not valid on Clifford algebras $\R_{p,q}$ with $p\ge2$: for example, the degree one polynomial 
$X(e_1-e_{12})+e_2-1$
has roots $e_{12}$ and $1/3e_1+2/3e_{12}$ in the quadratic cone, belonging to two distinct %sets $\cS_x$, $\cS_y$.
conjugacy classes.

The preceding theorem is valid in any Clifford algebra $\R_{0,q}$ for \emph{admissible} (see \cite{GhPe_AIM}) polynomials.  %(e.g.\ for polynomials with coefficients spanning a real vector subspace of $\QQ_{0,q}$ where the trace is real). %in the kernel of $\phi$). 
In particular, it holds for polynomials with para\-vector coefficients, a case considered in \cite{YangQianActa}.

A polynomial of degree $d$ can have more than $d$ roots if two or more of these are allowed to belong to the same conjugacy class. In $\R_{0,3}$, this can happen also for degree one polynomials. For example, the (not admissible) polynomial $X(e_1+e_{23})+1-e_{123}$ has two distinct roots $e_1$ and $e_{23}$ in the same conjugacy class $\cS$. Observe that $e_2\in\cS$ but is not a root of the polynomial.

Since the quadratic cone $\QQ_{0,3}$ contains the paravector space $\R^{(4)}$, one can consider the subset of paravector roots of  $P=\sum_{h=0}^dX^h a_h$:
\[V^{(4)}(P)=\{x\in\R^{(4)}\;|\; P(x)=\textstyle\sum_{h=0}^dx^h a_h=0\}\subset V(P).\]
Let $r$ be the number of real roots of $P$ (counted with multiplicity).
As shown in \cite{CoSaSt2009Israel}, if $V^{(4)}(P)$ contains two distinct roots in the same conjugacy class $\cS_y$, then $\cS_y\cap\R^{(4)}\subset V^{(4)}(P)$. We call these roots \emph{spherical roots} of $P$.  In this case, $P$ is divisible by the characteristic polynomial $\Delta_y(X)$ of $y$.
Let $s_y$ be the maximum exponent of a power of $\Delta_y$ dividing $P$. Let $s$ be the sum of integers $s_y$ when $y$ varies, without repetitions, in the conjugacy classes of non--real roots contained in $V^{(4)}(P)$. From Theorem~\ref{GordonMotzkin}, the number of these classes is at most $d$.
 We can get a more precise estimate, similar to what obtained in \cite{PogoruiShapiro} in the quaternionic case.

\begin{corollary}\label{GordonMotzkinParavectors}
Let $P(X)=\sum_{h=0}^dX^h a_h$ be a polynomial of positive degree $d$, with coefficients $a_h\in\R_{0,3}$. Let $r$ and $s$ be as before. Let $k$ be the number of non--real, non--spherical paravector roots of $P$. Then $r+2s+k\le d$.
\end{corollary}

\begin{proof}
We can factor out from $P$ a polynomial with real coefficients, of degree $r+2s$. The quotient is a polynomial of degree $d-r-2s$, to which Theorem~\ref{GordonMotzkin} applies. Therefore $k\le d-r-2s$, and the estimate is proved.
\end{proof}

\section{Main results}

\subsection*{Lagrange interpolation on the quaternionic space}

The problem of polynomial interpolation on quaternions has already been considered, usually studying properties of a quaternionic Vandermonde matrix. In \cite[\S16]{Lam1995} (in the general setting of division rings) and in \cite{RMXQLT1999}, it is proved that the problem has a unique solution if and only if the interpolation points are distinct and every conjugacy class contains at most two of the  points. Here we define the Lagrange interpolating polynomials, and give the supplementary condition that must be satisfied by the data to assure the existence of the solution when more than two points belong to the same conjugacy class.
This is a \emph{collinearity} condition involving also the values to be taken, coming from a property of polynomials with right coefficients (shared with the larger class of \emph{slice regular} functions, cf.~\cite{GeSt2007Adv,GhPe_AIM}): their restriction to each sphere $\cS_y$ is an affine function.

\begin{theorem}\label{LagrangeH}
Let $\cS_1,\ldots,\cS_n$ be pairwise distinct conjugacy classes of $\q$ and, for every $j \in \{1,\ldots,n\}$, let $x_{j1},\ldots,x_{jd_j}$ be pairwise distinct elements of $\cS_j$ with $d_j \geq 1$, let $w_{j1},\ldots,w_{jd_j}$ be arbitrary elements of $\q$ and let $d_j':=\min\{d_j,2\}$. Define $d:=-1+\sum_{j=1}^nd_j'$. Then there exists, and is unique, a quaternionic polynomial $P(X)=\sum_{h=0}^dX^h a_h$ of degree at most $d$ such that $P(x_{jk})=w_{jk}$ for each $j \in \{1,\ldots,n\}$ and $k \in \{1,\ldots,d_j\}$ if and only if, for every $j \in \{1,\ldots,n\}$ with $d_j \geq 3$, the following \emph{quaternionic collinear condition} $(\mc{C}_j)$ holds:
\[
(x_{j2}-x_{j1})^{-1}(w_{j2}-w_{j1})=(x_{jh}-x_{j1})^{-1}(w_{jh}-w_{j1})\quad \forall h \in \{3,\ldots,d_j\}.
\]
\end{theorem}

\begin{proof} 
Up to reordering indices, we can assume that $d_1=d_2=\ldots=d_m=1$ and $d_{m+1} \geq 2$, \dots, $d_n \geq 2$ for some $m \in \{0,1,\ldots,n\}$.

Given $T \in \q[X]$ and $y\in \q$, with $V(T)\cap\cS_y=\emptyset$, it follows from formula \eqref{product} that the polynomial  
\[S(X):=T(X)\punto(X-T(y)^{-1}y T(y))\]
vanishes exactly on $V(T) \cup \{y\}$. 
Then we can find, for each $j \in \{1,\ldots,m\}$ and for each $k \in \{m+1,\ldots,n\}$, polynomials $P_j$, $P_{k,1}$, $P_{k,2} \in  \q[X]$ such that
\[
V(P_j)=\{x_{11},\ldots,x_{(j-1)1},x_{(j+1)1},\ldots,x_{m1}\},
\]
\[
V(P_{k1})=\{x_{11},\ldots,x_{m1}\} \cup \{x_{k2}\},\quad
V(P_{k2})=\{x_{11},\ldots,x_{m1}\} \cup \{x_{k1}\}.
\]
For each $j \in \{1,\ldots,m\}$ and $k \in \{m+1,\ldots,n\}$, define the quaternionic Lagrange polynomials 
\begin{align*}
L_j(X)&:=\Delta_{x_{(m+1)1}}(X)\cdots\Delta_{x_{n1}}(X)\, P_j(X)\, a_j,
\\
L'_k(X)&:=\Delta_{x_{(m+1)1}}(X)\cdots\Delta_{x_{(k-1)1}}(X)\Delta_{x_{(k+1)1}}(X)\cdots\Delta_{x_{n1}}(X)\,P_{k1}(X)\, b_{k1},
\\
L''_k(X)&:=\Delta_{x_{(m+1)1}}(X)\cdots\Delta_{x_{(k-1)1}}(X)\Delta_{x_{(k+1)1}}(X)\cdots\Delta_{x_{n1}}(X) \,P_{k2}(X)\, b_{k2},
\end{align*}
where $a_j:=\left(\Delta_{x_{(m+1)1}}(x_{j1})\cdots\Delta_{x_{n1}}(x_{j1})P_j(x_{j1})\right)^{-1}$
and 
\[
b_{k\ell}:=\left(\Delta_{x_{(m+1)1}}(x_{k\ell})\cdots\Delta_{x_{(k-1)1}}(x_{k\ell})\Delta_{x_{(k+1)1}}(x_{k\ell})\cdots\Delta_{x_{n1}}(x_{k\ell})P_{k\ell}(x_{k\ell})\right)^{-1}
\]
with $\ell \in \{1,2\}$. Then
\[\textstyle
P:=\sum_{j=1}^mL_j  w_{j1}+\sum_{k=m+1}^n(L'_k  w_{k1}+L''_k  w_{k2})
\]
is an interpolating polynomial for points $\{x_{i1}\}_{i=1,\ldots, m}\cup\{x_{jh}\}_{j=m+1,\ldots, n}^{h=1,2}$. If there are conjugacy classes containing more than two data points, it remains to prove that $P$ interpolates also at the points $\{x_{jh}\}_{j=m+1,\ldots, n}^{h=3,\ldots,d_j}$. Since the restriction of $P$ to each sphere $\cS_y$ is an affine function, there exist  $a,b\in\q$ such that $P(x)=xa+b$  for every $x\in\cS_y$. For the sphere $\cS_{x_{jh}}$ the constants $a$ and $b$ are given by
\[a=(x_{j2}-x_{j1})^{-1}(w_{j2}-w_{j1}),\quad b=w_{j1}-x_{j1}a\,.
\]
Therefore,   the set of equalities $P(x_{jh})=x_{jh}a+b=w_{jh}$ for $h\in\{3,\ldots, d_j\}$, is equivalent to the collinearity condition ($\mc C_j$).

The uniqueness of the interpolating polynomial comes from the estimate on the number of roots of a quaternionic polynomial proved by Pogorui and Shapiro \cite{PogoruiShapiro}.
\end{proof}

We give an example of the procedure described in the proof for five points in $\q\times\q$ satisfying the collinearity conditions.

\begin{example}
Let $x_{11}=0$, $x_{21}=1+i$, $x_{31}=i$, $x_{32}=j$, $x_{33}=k$. Consider the values $w_{11}=1$, $w_{21}=-1$, $w_{31}=1$, $w_{32}=k$, $w_{33}=-j$. Note that $x_{31},x_{32},x_{33}$ belong to the same conjugacy class $\cS$, with characteristic polynomial $X^2+1$. The relative collinearity condition is satisfied:
\[(x_{32}-x_{31})^{-1}(w_{32}-w_{31})=-i=(x_{33}-x_{31})^{-1}(w_{33}-w_{31}).\]
We construct the Lagrange polynomials $L_1,L_2,L_3',L_3''$. We set:
\begin{align*}
P_1(X)&:=X-x_{21}=X-1-i\\
L_1^*(X)&:=(X^2+1)P_1(X)=X^3-X^2(1+i)+X-1-i\\
L_1(X)&:=L_1^*(X)L_1^*(x_{11})^{-1}=L_1^*(X)(-1-i)^{-1}=X^3\left(\tfrac{i-1}2\right)+X^2+X\left(\tfrac{i-1}2\right)+1\\
\text{and}\\
P_2(X)&:=X-x_{11}=X\\
L_2^*(X)&:=(X^2+1)P_2(X)=X^3+X\\
L_2(X)&:=L_2^*(X)L_2^*(x_{21})^{-1}=L_2^*(X)(-1+3i)^{-1}=X^3\left(\tfrac{-1-3i}{10}\right)+X\left(\tfrac{-1-3i}{10}\right).
\end{align*}
Moreover, let
\begin{align*}
Q(X)&:=(X-x_{11})\punto (X-x_{21})=X^2-X(1+i)\\
P_{31}(X)&:=Q(X)\punto (X-Q(x_{32})^{-1}x_{32}Q(x_{32}))=X^3+X^2\left(\tfrac{-3-i-j+2k}3\right)-X\left(\tfrac{-2+2i-3j+k}3\right)\\
L_3'(X)&:=P_{31}(X)P_{31}(x_{31})^{-1}=X^3\left(\tfrac{5-2j-k}{10}\right)+X^2\left(\tfrac{-1+k}2\right)+X\left(\tfrac{5-5i+3j-k}{10}\right)\\
\text{and}\\
P_{32}(X)&:=Q(X)\punto (X-Q(x_{31})^{-1}x_{31}Q(x_{31}))=X^3-X^2(1+2i)+X(-1+i)\\
L_3''(X)&:=P_{32}(X)P_{32}(x_{32})^{-1}=X^3\left(\tfrac{1-2i+2j+k}{10}\right)-X^2\left(\tfrac{1+k}2\right)+X\left(\tfrac{1+3i-3j+k}{10}\right).
\end{align*}
Finally, we get the unique interpolating polynomial of degree 3
\begin{align*}
P(X)&:=L_1(X)w_{11}+L_2(X)w_{21}+L_3'(X)w_{31}+L_3''(X)w_{32}=\\
&=L_1(X)-L_2(X)+L_3'(X)+L_3''(X)k=X^3i+X^2+1.
\end{align*}

\end{example}

\subsection*{Lagrange interpolation on the Clifford algebra $\R_{0,3}$}

We now perform the construction of Lagrange polynomials for a set of points in $\R_{0,3}$. Due to the presence of zero--divisors, we must restrict to points belonging to different conjugacy classes.

\begin{theorem}\label{Lagrange1}
Let $x_{1},\ldots,x_{m}$ be pairwise distinct elements of $\QQ_{0,3}$. Assume that $x_{j}\in\cS_j$ for each  $j \in \{1,\ldots,m\}$, with  $\cS_1,\ldots,\cS_m$ pairwise distinct conjugacy classes. Let $w_{1},\ldots,w_{m}$ be arbitrary elements of $\R_{0,3}$. Define $d:=m-1$. Then there exists a unique polynomial $P(X)=\sum_{h=0}^dX^h a_h$ with coefficients $a_h\in\R_{0,3}$, of degree at most $d$, such that $P(x_{j})=w_{j}$ for each $j \in \{1,\ldots,m\}$. 
\end{theorem}

In order to prove the theorem, we need a preliminary result.

\begin{lemma}\label{poly}
Given a polynomial $P(X)=\sum_{h=0}^dX^h a_h$ with coefficients $a_h\in\R_{0,3}$ and $y \in\QQ_{0,3}$, with $V(P)\cap\cS_y=\emptyset$ and $P(y)$ invertible, the polynomial $Q(X):=P(X)\punto(X-P(y)^{-1}y P(y))$ vanishes on $V(P) \cup \{y\}$, and $Q(x)$ is invertible for each $x\notin\cS_y$ such that $P(x)$ is invertible.
\end{lemma}
\begin{proof}
The first part follows from the equality 
\[Q(X)=X P(X)-P(X)P(y)^{-1}y P(y).\]
For each $x\in\QQ_{0,3}$ such that $P(x)$ is invertible, it holds
\[
Q(x)=P(x)(P(x)^{-1}xP(x)-P(y)^{-1}y P(y)).\]
If $x\notin\cS_y$, the latter equality and Proposition~\ref{invertible} give the invertibility of $Q(x)$.
\end{proof}

\begin{proof}[Proof of Theorem~\ref{Lagrange1}]
Let $P^{(1)}_m(X):=X-x_1$ and define recursively, for $k=2,\ldots,m-1$,  the polynomials 
\[P^{(k)}_m(X):=P^{(k-1)}_m(X)\punto (X-P_m^{(k-1)}(x_k)^{-1}x_k P_m^{(k-1)}(x_k)).
\]
Note that $P_m^{(k-1)}(x_l)$ is invertible for every $l=k,\ldots,m$, as can be seen applying inductively Lemma~\ref{poly}.
 The polynomial $P_m:=P_m^{(m-1)}$ vanishes at $x_1,\ldots,x_{m-1}$ and $P_m(x_m)$ is invertible.
We can then define the $m$-th Lagrange polynomial $L_m(X):=P_m(X)P_m(x_m)^{-1}$.

For each $j\in\{1,\ldots,m-1\}$, we can define similarly the Lagrange polynomials $L_1,\ldots, L_{m-1}$. Finally, we set $P(X):=\sum_{j=1}^m L_j(X)w_j$.

The uniqueness of the interpolating polynomial follows immediately from Theorem~\ref{GordonMotzkin}.
\end{proof}

As an illustration of the procedure described in the proof of Theorem~\ref{Lagrange1}, we give an example of Lagrange interpolation for three points in $\R_{0,3}\times\R_{0,3}$. 

\begin{example}
Let $x_1=e_1$, $x_2=e_2+e_{23}$, $x_3=-1$ in $\QQ_{0,3}$ and let $w_1=1$, $w_2=2e_{23}$, $w_3=e_1$. We construct the Lagrange polynomial $L_3$:
\begin{align*}
P_3^{(1)}(X)&:=X-e_1\\
P_3(X)&:=P_3^{(2)}(X):=(X-e_1)\punto (X-(e_2+e_{23}-e_1)^{-1}(e_2+e_{23})(e_2+e_{23}-e_1))=\\
&=X^2+X\left(\frac{e_1}{5}-\frac{3 e_2}{5}+\frac{2 e_1 e_3}{5}-\frac{e_2 e_3}{5}\right)+\left(\frac{6}{5}+\frac{3 e_1 e_2}{5}+\frac{2 e_3}{5}+\frac{1}{5} e_1 e_2 e_3\right)\\
L_3(X)&:=P_3(X)P_3(-1)^{-1}=X^2a_3^2+Xa_3^1+a_3^0, \quad\text{with}\\
a_3^0&=\tfrac1{30}(16+4 e_1-3 e_2+3 e_1 e_2+2 e_3+2 e_1 e_3+e_2 e_3+e_1 e_2 e_3),\\
a_3^1&=\tfrac1{30}(-3+5 e_1-6 e_2+4 e_1 e_3+2 e_1 e_2 e_3),\\
a_3^2&=\tfrac1{30}(11+e_1-3 e_2-3 e_1 e_2-2 e_3+2 e_1 e_3-e_2 e_3+e_1 e_2 e_3).
\end{align*}
Similarly, we compute the Lagrange polynomials $L_1$ and $L_2$:
\begin{align*}
L_1(X)&:=X^2a_1^2+Xa_1^1+a_1^0, \quad\text{with}\\
a_1^0&=\tfrac1{20}(8-8 e_1+6 e_2-6 e_1 e_2-4 e_3-4 e_1 e_3-2 e_2 e_3-2 e_1 e_2 e_3),\\
a_1^1&=\tfrac1{20}(6-10 e_1+12 e_2-8 e_1 e_3-4 e_1 e_2 e_3),\\
a_1^2&=\tfrac1{20}(-2-2 e_1+6 e_2+6 e_1 e_2+4 e_3-4 e_1 e_3+2 e_2 e_3-2 e_1 e_2 e_3),\\
\text{and\quad}L_2(X)&:=X^2a_2^2+Xa_2^1+a_2^0, \quad\text{with}\\
a_2^0&=\tfrac1{15}(1+4 e_1-3 e_2+3 e_1 e_2+2 e_3+2 e_1 e_3+e_2 e_3+e_1 e_2 e_3),\\
a_2^1&=\tfrac1{15}(-3+5 e_1-6 e_2+4 e_1 e_3+2 e_1 e_2 e_3),\\
a_2^2&=\tfrac{1}{15} (-4+e_1-3 e_2-3 e_1 e_2-2 e_3+2 e_1 e_3-e_2 e_3+e_1 e_2 e_3).
\end{align*}
Finally, we get the interpolating polynomial 
\begin{align*}
P(X)&:=\sum_{j=1}^m L_j(X)w_j:=X^2a_2+Xa_1+a_0, \quad\text{with}\\
a_0&=\tfrac{1}{15} (6 e_1+8 e_2-5 e_1 e_2-10 e_3-2 e_1 e_3+4 e_2 e_3+5 e_1 e_2 e_3),\\
a_1&=\tfrac1{15}(-6-9 e_1+13 e_2+3 e_1 e_2-8 e_3-12 e_1 e_3-e_2 e_3+9 e_1 e_2 e_3),\\
a_2&=\tfrac1{15}(-6+5 e_2+8 e_1 e_2+2 e_3-10 e_1 e_3-5 e_2 e_3+4 e_1 e_2 e_3).
\end{align*}
Observe that the expansion of $P(X)$ w.r.t.\ the eight real coordinates contains 239 terms.

\end{example}

%\bibliographystyle{plain}
%\bibliography{Ref_AP}{}

\begin{thebibliography}{10}

\bibitem{CSSSbook}
F. Colombo, I. Sabadini, F. Sommen, and D.C. Struppa.
\newblock {\em Analysis of {D}irac systems and computational algebra},
  volume~39 of {\em Progress in Mathematical Physics}.
\newblock (Birkh\"auser Boston Inc., Boston, MA, 2004).

\bibitem{CoSaSt2009Israel}
F. Colombo, I. Sabadini, and D.C. Struppa.
\newblock Slice monogenic functions.
\newblock {\em Israel J. Math.} {\bf 171} (2009) 385--403.

\bibitem{GeSt2007Adv}
G. Gentili and D.C. Struppa.
\newblock A new theory of regular functions of a quaternionic variable.
\newblock {\em Adv. Math.} {\bf 216(1)} (2007) 279--301.

\bibitem{GhPe_Trends}
R. Ghiloni and A. Perotti.
\newblock A new approach to slice regularity on real algebras.
\newblock In {\em Hypercomplex analysis and its Applications}, Trends Math.,
  pages 109--124. (Birkh\"auser, Basel, 2011).

\bibitem{GhPe_AIM}
R. Ghiloni and A. Perotti.
\newblock Slice regular functions on real alternative algebras.
\newblock {\em Adv. Math.}  {\bf 226(2)} (2011) 1662--1691.

\bibitem{GordonMotzkin}
B.~Gordon and T.~S. Motzkin.
\newblock On the zeros of polynomials over division rings.
\newblock {\em Trans. Amer. Math. Soc.}  {\bf 116} (1965) 218--226.

\bibitem{GHS}
K. G{\"u}rlebeck, K. Habetha, and W. Spr{\"o}{\ss}ig.
\newblock {\em Holomorphic functions in the plane and {$n$}-dimensional space}.
\newblock (Birkh\"auser Verlag, Basel, 2008).

\bibitem{HHA}
E. Hitzer, J. Helmstetter, and R. Ablamowicz.
\newblock Square roots of $-1$ in real clifford algebras.
\newblock In {\em Quaternion and Clifford Fourier transforms and wavelets},
  Trends Math. (Birkh\"auser, Basel, 2013).

\bibitem{RMXQLT1999}
R.M. Hou, X.Q. Zhao, and L.T. Wang.
\newblock The double determinant of {V}andermonde's type over quaternion field.
\newblock {\em Appl. Math. Mech.} {\bf 20(9)} (1999) 977--984.

\bibitem{Lam}
T.~Y. Lam.
\newblock {\em A first course in noncommutative rings}, volume 131 of {\em
  Graduate Texts in Mathematics}.
\newblock (Springer-Verlag, New York, 1991).

\bibitem{Lam1995}
T.~Y. Lam.
\newblock {\em Exercises in classical ring theory}.
\newblock Problem Books in Mathematics. (Springer-Verlag, New York, 1995).

\bibitem{PogoruiShapiro}
A.~Pogorui and M.~Shapiro.
\newblock On the structure of the set of zeros of quaternionic polynomials.
\newblock {\em Complex Var. Theory Appl.} {\bf 49(6)} (2004) 379--389.

\bibitem{YangQianActa}
Y. Yang and T. Qian.
\newblock On sets of zeroes of {C}lifford algebra--valued polynomials.
\newblock {\em Acta Math. Sci. Ser. B Engl. Ed.} {\bf 30(3)} (2010) 1004--1012.

\end{thebibliography}

%%%%%%%%%%%%%%%%%%%%%

\end{document}